\documentclass[12pt,a4paper]{article}
\usepackage{amsmath,amssymb,amsthm,mathptmx}
\usepackage[english,russian]{babel}
\usepackage{amsfonts, dsfont, graphicx, srcltx}
\usepackage{mathrsfs}
\newtheorem{theorem}{Теорема}

\newtheorem{remark}{Замечание}

\author{Nikita V. Gaianov, Anastasia V. Parusnikova}
\title{On finding formal power-logarithmic expansions of solutions to $q$-difference equations}%\thanks{This study was carried out within The National Research University Higher School of Economics Academic Fund Program in 2013-2014, research grant No. 12-01-0030.}}

\date{}

\begin{document}
\maketitle

\begin{abstract}
	
	An algebraic $q$-difference equation is considered. A sufficient condition for the existence of a formal power-logarithmic expansion of a solution to such an equation in the neighborhood of zero is proposed. An example of applying this sufficient condition for constructing a formal expansion of a solution to a certain $q$-difference analogue of the fifth Painlev\'{e} equation for specific values of the equation parameters is given; two different values of the number $q$ are considered, leading to qualitatively different formal asymptotic expansions of the solutions of the fifth Painlev\'{e} equation.

\textbf{Keywords:} $q$-difference equations, asymptotic expansions, Newton polygon, power-logarithmic expansion.

\textbf{ MSC classes:} 	34m25, 34m55.

Рассматривается алгебраическое $q$-разностное уравнение. Предлагается достаточное условие существования формального степенно-логарифмического разложения решения такого уравнения в окрестности нуля. Приводится пример применения этого достаточного условия для построения формального разложения решения некоторого $q$-разностного аналога пятого уравнения Пенлеве при конкретных значениях параметров уравнения; рассматриваются два различных значения числа $q$, приводящие к качественно разным формальным асимптотическим разложениям решений.

\textbf{Ключевые слова:} $q$-разностные уравнения, асимптотические разложения, многоугольник Ньютона, степенно-логарифмическое разложение.

%\textbf{ MSC classes:} 	34m25, 34m55.

%Methods of French and Japaneese schools are used to analyse these properties of the solutions. The results obtained are compared with the ones obtained by means of Power Geometry.

    %
\end{abstract}

	\section{Введение}
В настоящее время активно развивается теория $q$-разностных уравнений  и систем~\cite{crystal, vyugin}. %qP3, vyugin}. 
Хорошо освоены методы нахождения решений в виде формальных степенных 
разложений~\cite{canopoly}, имеются доказательства их сходимости \cite{gontsov}, однако даже для линейных 
дифференциальных уравнений могут встречаться решения, содержащие логарифмические 
слагае\-мые~\cite{adams}. В предыдущей статье авторов \cite{us} было  получено достаточное условие 
существования решения в виде степенно-логарифмического разложения по целым 
неотрицательным степеням независимой переменной (в виде ряда Дюлака). Развивая эти результаты, в данной 
работе мы переносим методы и результаты степенной геометрии \cite{bruno2004, brunogreen} на случай $q$-разностных 
уравнений, формулируем достаточные условия существования формальных решений алгебраического $q$-разностного уравнения в виде 
степенно-логарифмических рядов более общего вида, чем в предыдущей работе, и 
представляем метод их получения. Также приведён пример применения полученной теоремы для построения формального разложения решения некоторого $q$-разностного аналога пятого уравнения Пенлеве при конкретных значениях параметров уравнения.   Рассматриваются два различных значения числа $q$, приводящие к качественно разным формальным асимптотическим разложениям решения аналога пятого уравнения Пенлеве.

\section{Основные определения}
В этой работе 
рассматривается случай $x \to 0$. 

Предположим, что $y$ -- однозначная комплекснозначная функция комплексной переменной $x$, число $q \neq 0$. Определим \textit{оператор $q$-дифференцирования} $\sigma$ формулой:  
$$
(\sigma y)(x)=y(q x),
$$
при этом функция $\sigma y$ называется  \textit{$q$-разностной производной функции} $y$.

Далее также предполагаем, что $q^k \neq 1~\forall k \in \mathbb{N}$.

\textit{Алгебраическим $q$-разностным уравнением $n$-го порядка} называется уравнение:
\begin{equation}
f(x, y, \sigma y, \ldots, \sigma^n y) = 0, \label{eq:q_alg_diff}
\end{equation}
где $f$ -- многочлен $n+2$ переменных.

Перенесем имеющиеся для  дифференциальных уравнений определения степенной геометрии из  работы \cite{bruno2004} для построения аналогичной теории для $q$-разностных уравнений.

Обозначим $X = (x, y)$.
\textit{$q$-разностным мономом} $b(x, y)$ называется произведение монома $c 
x^{r_1} y^{r_2}$,  
где $c = \text{const}$, $r_1, r_2 \in \mathbb{R}$, и конечного числа $q$-разностных 
производных  
$$\sigma^l y, \quad l \in \mathbb{N}.$$
\textit{$q$-разностной суммой} называется сумма $q$-разностных мономов:  
\begin{equation}
f(X) = \sum a_i(X). \label{eq:difsum}
\end{equation}
Каждому $q$-разностному моному $b(X)$ ставится в соответствие векторный показатель 
$Q(b)$ следующим образом: 
$$Q(c x^{r_1} y^{r_2}) = (r_1, r_2);~~~
Q((\sigma^l y)^r) = (0,r);~~~ Q(b_1(X) b_2(X)) = Q(b_1(X)) + Q(b_2(X)).$$
Обозначим как $S(f)$ 
множество векторных показателей $q$-разностной суммы $f(X)$, а через $f_Q(X)$ -- сумму 
всех мономов $b_i$, для которых $Q(b_i) = Q$. Тогда  
$$f(X) = \sum_{Q \in S(f)} f_Q(X).$$  
Множество $S(f)$ называется \textit{носителем} суммы $f(X)$ и уравнения 
\eqref{eq:q_alg_diff}.  

Заметим, что носитель алгебраического $q$-разностного уравнения $(\ref{eq:q_alg_diff})$  лежит в~$\mathbb{Z}_+^2.$

\textit{Многоугольником Ньютона} уравнения \eqref{eq:q_alg_diff} (а также $q$-разностной суммы 
\eqref{eq:difsum}) называется выпуклая оболочка множества $S(f)$, которая обозначается 
$\Gamma(f)$.  
Граница $\partial \Gamma(f)$ многоугольника $\Gamma(f)$ состоит из вершин 
$\Gamma_j^{(0)}$ и рёбер $\Gamma_j^{(1)}$ -- (обобщённых) граней 
$\Gamma_j^{(d)}$, где верхний индекс $d$ указывает размерность грани.  
Каждой обобщённой грани соответствует граничное подмножество  
$$S_j^{(d)} = S(f) \cap \Gamma_j^{(d)},$$
\textit{укороченная сумма}, определяемая как  
\begin{equation*}
\hat{f}_j^{(d)}(X) = \sum_{Q \in S_j^{(d)}} f_Q(X)
\end{equation*}  
и \textit{укороченное уравнение} 
\begin{equation}
\label{tr_eq}
f_j^{(d)}(X) = 0.
\end{equation}

\section{Связь решений исходного и укороченного уравнения}

Если для уравнения \eqref{eq:q_alg_diff} существует решение вида
\begin{equation}
y = c x^r + O(x^{r+\varepsilon}), r\in \mathbb{R}, x \to 0, \varepsilon > 0, c \in \mathbb{C}\setminus\{0\}, 
\label{eq:solution}
\end{equation}
то его \textit{укороченным решением} называется
\begin{equation}
y = c x^r, \label{eq:shortsolution}
\end{equation}
а \textit{нормальным конусом решения} \eqref{eq:solution} называется луч $\lambda (-1, -r)$, где $\lambda > 0$.

\textit{Нормальным конусом} грани $\Gamma_j^{(d)}$ называется множество 
$$
U_j^{(d)}=\left\{\begin{array}{ccc}
P: & \langle P, Q\rangle=\left\langle P, Q^{\prime}\right\rangle, & Q, Q^{\prime} \in {S}_j^{(d)} \\
& \langle P, Q\rangle>\left\langle P, Q^{\prime \prime}\right\rangle, & Q^{\prime \prime} \in {S}(f) 
\backslash {S}_j^{(d)}
\end{array}\right\},
$$
где $\langle P , Q \rangle := p_1 q_1 + p_2 q_2$ -- скалярное произведение векторов $P=(p_1, p_2)$, $Q=(q_1, q_2)$.

\begin{theorem}
Если уравнение \eqref{eq:q_alg_diff} имеет решение \eqref{eq:solution} и нормальный 
конус $U$ решения \eqref{eq:solution} таков, что $U \subset U_j^{(d)}$, то укороченное решение \eqref{eq:shortsolution} является решением 
соответствующего	укороченного уравнения \eqref{tr_eq}.
\end{theorem}

\begin{proof}
Подставим решение (\ref{eq:solution}) в уравнение (\ref{eq:q_alg_diff}):
\begin{multline*}
	f(X) = \sum_{Q\in S(f)} f_Q(x, c x^r + O(x^{r+\varepsilon})) =  \sum_{Q \in S_j^{(d)}} f_Q(x, c x^r) + 
	O(x^{r'}) + O(x^{r''}) = 0,
\end{multline*}
где $r' = \min\limits_{Q \in S_j^{(d)}} \langle Q, (1, r+\varepsilon) \rangle$,
$r'' = \min\limits_{ Q \in S(f) \setminus S_j^{(d)}} \langle Q, (1, r) \rangle$. 

Но поскольку $U \subset U_j^{(d)}$, 
то $r'' > \langle Q, (1, r) \rangle$ для всех $Q \in S_j^{(d)}$. Получаем что 
$$
f(X) = \sum_{Q \in S_j^{(d)}} f_Q(x, c x^r)(1 + o(1)), x \to 0,
$$
и для выполнения равенства требуется
$$\sum_{Q \in S_j^{(d)}} f_Q(x, c x^r) = 0,$$
что и представляет собой укороченное уравнение \eqref{tr_eq}.
\end{proof}

\section{Решение укороченных уравнений}

Рассмотрим укороченное уравнение $\hat{f}^{(0)}=0$, соответствующее вершине $\Gamma^{(0)}=(q_1, q_2)$ многоугольника Ньютона.
При подстановке в него $y = c x^r$ и сокращении степеней $x$ и $c$ получаем уравнение
$$c^{-q_2} x^{-q_1-q_2 r} \hat{f}^{(0)}(x, c x^r) = \chi(r) = 0,$$
которое зависит только от $r$ и, вообще говоря, от $q$.  
Многочлен $\chi(r)$ называется \textit{характеристическим многочленом $q$-разностной суммы}
$f^{(0)}(X)$.  
Из его корней необходимо отобрать те, для которых вектор $(-1, -r)$ лежит в нормальном 
конусе $U^{(0)}$.

Рассмотрим укороченное уравнение, соответствующее ребру $\Gamma^{(1)}$, лежащему на прямой $q_1+ r q_2+c=0.$  Для того чтобы 
решение $y = c x^r$ было решением укороченного уравнения $\hat{f}^{(1)}(x, y)=0$, необходимо, чтобы $(-1, -r) \in 
U^{(1)}$, что однозначно определяет значение $r$. Значение $c_r$ находится из 
\textit{определяющего уравнения}:
$$ x^{c} \hat{f}_j^{(1)}(x, c_r x^r) = 0.$$

Итак, каждое укороченное уравнение имеет одно или несколько подходящих решений с $U \subset 
U^{(d)}$. 

\section{Критические числа укороченного уравнения}

Если найдено укороченное решение \eqref{eq:shortsolution}, то замена $y = c x^r + z$ 
приводит уравнение \eqref{eq:q_alg_diff} к виду
\begin{equation}
\tilde{f}(x, z) = f(x, c x^r + z) = 0.
\label{eq:neweq}
\end{equation}

Во многих случаях уравнение \eqref{eq:neweq}, возможно, после сокращения на некоторую 
степень $x$, имеет вид:
\begin{equation}
\label{subs}
\tilde{f}(x, z): = \mathcal{L}(\sigma) z + h(x, z) = 0,
\end{equation}
где $\mathcal{L}(\sigma)$ — линейный $q$-разностный оператор с постоянными 
коэффициентами, т.~е.
\begin{equation}
\label{L-oper}
\mathcal{L}(\sigma) = a_m \sigma^m + \ldots + a_1 \sigma + a_0,
\end{equation}
точка $Q\left(\mathcal{L}(\sigma)z\right) = (0, 1)$ присутствует в носителе уравнения (\ref{subs}) и  является вершиной многоугольника $\Gamma(\tilde{f})$, а носитель $S(h)$  не 
содержит точки $(0, 1)$.

Определим \textit{характеристический многочлен} $q$-разностной суммы $\mathcal{L}(\sigma) z$ по формуле
$$\nu(k) = x^{-k} \mathcal{L}(\sigma)[x^k].$$
Если $\nu(k) \not\equiv 0,$ то корни $k_1, \ldots, k_s$ многочлена $\nu(k)$ называются \textit{собственными числами} 
укороченного решения \eqref{eq:shortsolution}.
Вещественные собственные числа $k\in \{k_1, \ldots, k_s\}$, для которых $k > r$, называются \textit{критическими 
числами}.

Введём дополнительные обозначения. Сдвинем носитель $S(\tilde{f})$ на $(0, 
-1)$ и будем обозначать новое множество $S'(\tilde{f}) = S(\tilde{f}) - (0, 1)$. 
Пусть задано такое число $r$, что для каждой точки $Q' \in S'(\tilde{f})$ скалярное 
произведение $\langle R, Q' \rangle \leq 0$, где $R=(1, r)$. 
%Также предположим, что критические числа $k\in \{k_1, \ldots, k_s\}$ удовлетворяют неравенству $k>r$. 
Обозначим через $S'_+(\tilde{f})$ множество конечных сумм точек $Q' \in S'(\tilde{f})$ и векторов $(k_1, 
-1), \ldots, (k_s, -1)$, где $k_1, \ldots, k_s$ -- критические числа укороченного решения $y=cx^r$. 
Пусть $K(k_1, \ldots, k_s)$ -- множество таких $q_1\in\mathbb{R}$, что $(q_1, -1) \in 
S'_+(\tilde{f})$.

\section{Степенно-логарифмические разложения решений}

Как показано в статье \cite{us}, не для каждого алгебраического $q$-разностного уравнения все решения являются степенными, т. е. элементами  пространства $\mathbb{C}[[x]] = \left\{\sum\limits_{k=0}^\infty c_k x^k, c_k =\mathrm{const}\in \mathbb{C}\right\}.$ В работе \cite{us} мы ограничивались формальными решениями в виде рядов
$$
\sum_{k=0}^\infty p_k(\log_q x)x^k, 
$$ где $p_k$ -- многочлены с комплексными коэффициентами. 

В данной 
работе рассматриваем формальные ряды несколько более общего вида~(\ref{eq:logsolution})~-- \textit{степенно-логарифмические асимптотические разложения решений}.

Далее приведено доказательство аналогичного имеющемуся в \cite{bruno2004} для дифференциального уравнения достаточного условия существования формального степенно-логарифмического асимптотического разложения решения $q$-разностного алгебраического уравнения.

\begin{theorem}
\label{th}
Рассмотрим уравнение \eqref{eq:q_alg_diff} и его укороченное решение \eqref{eq:shortsolution}. Пусть, сделав замену $y=cx^r+z$  и преобразовав, получили уравнение \eqref{subs} и пусть для носителя уравнения \eqref{subs} выполнены следующие условия:
\begin{enumerate}
	\item[(1)] точка $(0, 1)$ является вершиной $\Gamma(\tilde{f})$; 
	\item[(2)] в уравнении $\tilde{f}(x, z) = 0$ вершине $(0, 1)$ 
	соответствует слагаемое $\mathcal{L}(\sigma) z$ и только оно, где $\mathcal{L}(\sigma)$ -- линейный $q$-разностный оператор, определяемый формулой \eqref{L-oper} ($\mathcal{L}(\sigma)$ -- многочлен от $\sigma$, его коэффициенты являются постоянными).
\end{enumerate}

Тогда	 уравнение \eqref{subs} имеет 
формальное решение вида
\begin{equation}
	z = \sum_k \beta_k(\log_q x) x^k, \quad k \in K(k_1, \ldots, k_s), k> 
	r,\label{eq:logsolution}
\end{equation}
где $\beta_k$ — многочлены от переменной $\log_q x$, а $k_1, \ldots, k_s$ — критические 
числа укороченного решения \eqref{eq:shortsolution}.
\end{theorem}
\begin{proof}
Представим уравнение \eqref{subs} в виде
$$\mathcal{L}(\sigma) z = - h(x, z).$$
Подставим формальное решение \eqref{eq:logsolution} в уравнение \eqref{subs}. После подстановки в левой части   
содержатся слагаемые со степенями $x^k$, где $k \in K(k_1, \ldots, k_s)$, $k>r$. В правой части будут 
слагаемые со степенями вида $x^{\langle Q, (1, k) \rangle}$, где $Q = (q_1, q_2) \in S(\tilde{f})$. 

Но число $\langle Q, (1, k) \rangle \in K(k_1, \ldots, k_s)$, поскольку $\langle Q, (1, k) \rangle = 
q_1 + k q_2$. В множестве $K(k_1, \ldots, k_s)$ содержится абсцисса вершины $Q' + q_2 (k, -1) = (q_1, q_2 - 1) + q_2 (k, 
-1) = (q_1 + k q_2, -1)$ (тут мы пользуемся тем, что $q_2\in \mathbb{Z}_+$, т.~е. точку $(k, -1)$ складываем с собой конечное число раз). 

Следовательно, приравнивая степени при одинаковых степенях $x^k$, получаем семейство 
разностных уравнений на $\beta_k$, при этом здесь и далее $t=\log_q x$:
\begin{equation}
	\label{theta}	
	\mathcal{L}(q^k T) \beta_k(t) +\theta_k(t) = 0,
\end{equation}
где $\theta_k$ — многочлен от функций $\beta_\ell,$ где $\ell \in K(k_1, \ldots, k_s)$, а также $\ell < k$; 
оператор сдвига $T$ определяется по формуле $(Tf)(t)= f(t+1).$
\end{proof}
\begin{remark}
В теореме $\ref{th}$ получены уравнения \eqref{theta}, аналогичные имеющимся в \cite{us} уравнениям, но теперь индексы $k$ не только  целые неотрицательные.
\end{remark}

Говорят, что для критического числа $k$ \textit{выполнено условие совместности}, если в уравнении~(\ref{theta}) функция $\theta_k(t)\equiv 0$.

\begin{remark}
В условиях теоремы $\ref{th}$, пусть для всех критических чисел $k$ выполнено условие совместности, а также все критические числа $k$ являются некратными. Тогда формальное решение \eqref{eq:logsolution} уравнения \eqref{eq:neweq} не содержит логарифмов.%, либо имеется кратное критическое число. 
\end{remark}

\begin{proof}

Функции $\beta_k$ получаются как решения разностных уравнений (\ref{theta}).

Если $k$ является некратным критическим, и для него выполнено условие совместности, то уравнение на $\beta_k$ имеет вид
$$\mathcal{L}(q^k T)\beta_k(t) = 0,$$
тогда решение $\beta_k(t)=C_k,$ $C_k \in \mathbb{C}$, -- произвольная константа.

Пусть $k$ не является критическим. Функция $\theta_k(t)=A_k$ в уравнении (\ref{theta}) -- постоянна, уравнение на $\beta_k$ имеет вид
$$\mathcal{L}(q^k T)\beta_k(t) =  -A_k=\mathrm{const},$$  
тогда его решение $\beta_k(t)=-A_k/\mathcal{L}(q^k)$ -- однозначно определенная постоянная. 
\end{proof}

\section{Степени логарифмов в разложении}

Посмотрим, как растут степени логарифмов в разложении \eqref{eq:logsolution}.

Обозначим через $\mu(j)$ кратность $q^j$ как корня $\mathcal{L}(s)$. 

\begin{theorem}
Пусть выполнены условия теоремы $\ref{th}$. %Если в уравнении \eqref{eq:neweq} отсутствует логарифмы, то 
Тогда	для степеней многочленов $\beta_k$ в решении \eqref{eq:logsolution} выполняются оценки
\begin{equation}
	\label{ineq}
	\deg \beta_k \leq C (k - r) \sum_{r<j\leq k}\mu(j)
\end{equation}
при $C=1+\max\limits_{\ell\in K(k_1, \ldots, k_s)\cap \{\ell>r\}} \frac{1}{\ell-r}$.
\end{theorem}
\begin{proof}
Отдельно отметим, что здесь считаем степень нулевого многочлена равной нулю.

Докажем по индукции. Рассмотрим $k^{(0)} = \min \left(K(k_1, \ldots, k_s)\cap \{k>r\}\right)$ -- минимальную 
степень в разложении \eqref{eq:logsolution}. 
Тогда 
$\deg \beta_{k^{(0)}} \leq \mu(k^{(0)})$, и, поскольку $C>\frac{1}{k^{(0)} - r}$, неравенство (\ref{ineq})
выполнено.

Предположим, что неравенство (\ref{ineq}) верно для всех $\beta_j$  при $j < k$. Для $\beta_k$ имеем 
уравнение (\ref{theta}), в котором $\theta_k$ -- линейная комбинация мономов вида
\begin{multline}
	\label{monom}
	\beta_{i_{0,0}}^{\alpha_{0,0}}(t) \ldots  \beta_{i_{0,N_0}}^{\alpha_{0,N_0}}(t) \, 
	\beta_{i_{1, 0}}^{\alpha_{1,0}}(t+1) \ldots  \beta_{i_{1,N_1}}^{\alpha_{1,N_1}}(t+1)  \ldots       
	\beta_{i_{n,0}}^{\alpha_{n,0}}(t+n) \ldots  \beta_{i_{n,N_n}}^{\alpha_{n,N_n}}(t+n), 
\end{multline}
где $\alpha_{0,0}i_{0,0} + \ldots+ \alpha_{n,N_n}i_{n, N_n} \leq k$.
Из последнего неравенства также следует, что 
$$\alpha_{0,0}(i_{0,0}-r) + \ldots + \alpha_{n,N_n}(i_{n,N_n}-r) \leq k - r.$$

Степень каждого монома (\ref{monom}) по предположению индукции не превышает
\begin{multline*}
	\alpha_{0,0}\deg \beta_{i_{0,0}} + \ldots + \alpha_{n, N_n}\deg \beta_{i_{n, N_n}} \leq \\ \leq C\left(\alpha_{0,0} 
	(i_{0,0}-r)\sum_{r<j\leq i_{0,0}}\mu(j) + 
	\ldots + \alpha_{n, N_n}(i_{n, N_n}-r)\sum_{r<j\leq i_{n,Nn}}\mu(j)\right)  \leq \\ \leq C(k-r) 
	\sum_{r<j<k}\mu(j).
\end{multline*}
Тогда
$$
\deg \beta_k \leq \deg\theta_k + \mu(k) 
\leq C(k-r) 
\sum_{r<j<k}\mu(j) 
+ \mu(k) \leq C(k-r) \sum_{r<j\leq k}\mu(j).
$$
Последнее неравенство верно в силу того, что $C(k-r) > 1.$
\end{proof}

\section{Пример степенно-логарифмического разложения.}

Рассмотрим пятое уравнение Пенлеве:	
$$
\frac{d^2 y}{dx^2} =
\left( \frac{1}{2y} + \frac{1}{y-1} \right) \left( \frac{dy}{dx} \right)^2
- \frac{1}{x} \frac{dy}{dt}
+ \frac{(y-1)^2}{x^2} \left( a_1 y + \frac{a_2}{y} \right)
+ a_3 \frac{y}{x}
+ a_4 \frac{y(y+1)}{y-1}
,$$
где $a_1, a_2, a_3, a_4$ -- комплексные параметры. Положив $a_1=a_2=0, a_3, a_4 \neq 0$ и перейдя к оператору $\delta = x\frac{d}{dx}$, получим 
уравнение
$$
\frac{\delta^2 y - \delta y}{x^2} =
\left( \frac{1}{2y} + \frac{1}{y-1} \right) \frac{(\delta y)^2}{x^2}
- \frac{1}{x^2} \delta y
+ a_3 \frac{y}{x}
+ a_4 \frac{y(y+1)}{y-1}
.$$
Формально заменим в этом уравнении оператор $\delta$ на $\sigma$ и получим некоторое $q$-разностное  
уравнение
$$
\frac{\sigma^2 y - \sigma y}{x^2} =
\left( \frac{1}{2y} + \frac{1}{y-1} \right) \frac{(\sigma y)^2}{x^2}
- \frac{1}{x^2} \sigma y
+ a_3 \frac{y}{x}
+ a_4 \frac{y(y+1)}{y-1}.
$$
Упростим его, домножив на $x^2 y(y-1)$, получим уравнение
\begin{equation}
-a_3xy^3+a_3xy^2-a_4x^2y^3-a_4x^2y^2+(\sigma^2y)y^2-\frac{3(\sigma y)^2y}{2}-(\sigma^2 
y)y+\frac{(\sigma y)^2}{2} = 0. \label{eq:qP5}
\end{equation}
Многоугольник Ньютона $\Gamma$ уравнения (\ref{eq:qP5}) изображен на Рис. 1.

\begin{figure}[h]
\centering
\includegraphics{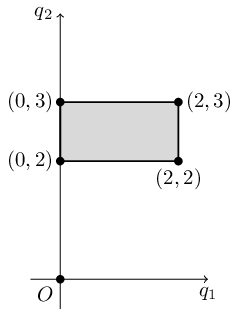}
\caption{Многоугольник Ньютона уравнения \eqref{eq:qP5}.}
\end{figure}

Нас интересует левое вертикальное ребро многоугольника $\Gamma$. Ему соответствует определяющее уравнение
$$
\frac{c^2}{2}-\frac{3c^3}{2}-c^2+c^3 = 0,$$
единственное ненулевое решение которого есть $c=-1.$
Произведем замену $y = z -1$ в уравнении \eqref{eq:qP5} и перейдем к уравнению
\begin{multline}
\label{eq:qP5-new}
-a_3xz^3+4a_3xz^2-5a_3xz+2a_3x-a_4x^2z^3+2a_4x^2z^2-a_4x^2z+(\sigma^2 
z)z^2-z^2-\\-\frac{3}{2}(\sigma z)^2z+3(\sigma z)z-3(\sigma^2 z)z+\frac{3z}{2}+2(\sigma 
z)^2-4(\sigma z)+2(\sigma^2 z) = 0, 
\end{multline}
многоугольник Ньютона которого изображен на Рис. 2.
\begin{figure}[h]
\centering
\includegraphics{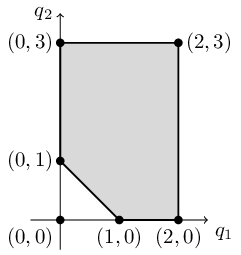}
\caption{Многоугольник Ньютона уравнения \eqref{eq:qP5-new}.}
\end{figure}

Линейная часть уравнения (\ref{eq:qP5-new}) имеет вид
$\mathcal{L}(\sigma)z = 2\sigma^2 z - 4 \sigma z + 3z/2,$ откуда 
характеристическое уравнение $$\nu(k) = 2 q^{2k} - 4 q^k + 3/2 = 0,$$ корни $k$ которого определяются из уравнения $q^k = 1/2$  или  
$	q^k=3/2.$

Продолжим укороченное решение $y=-1$ до формального разложения решения уравнения \eqref{eq:qP5}. 

Рассмотрим несколько значений параметра $q$.

\textbf{Случай 1.} Пусть сначала $q = 1/2,$ тогда собственные  числа $k_1 = 1, k_2 = 1-\log_{2} 3$. Только $k_1$ является критическим числом ($k_2<0$).

Поскольку координаты всех точек носителя уравнения (\ref{eq:qP5-new}) целые 
неотрицательные, критическое число натуральное, а в носителе нет точки $(0, 0)$, то и множество $K(
k_1) \subset \mathbb{N}$. C другой стороны, в множестве $S'_+(\tilde{f})$  есть точки $(0, 1)$ и $(1, -1)$, 
их 
конечные суммы покрывают все числа вида $(n, -1), n\in\mathbb{N}$, значит, $K(
k_1) = \mathbb{N}$.
Итак, уравнение \eqref{eq:qP5-new} имеет степенно-логарифмическое 
решение с носителем, лежащим в $\mathbb{N}$.

Перейдём к нахождению решения.
Второе слагаемое в разложении удовлетворяет уравнению
$$2\sigma^2z - 4\sigma z + \frac{3}{2}z + 2 a_3 x = 0.$$
Подставляя в уравнение выше решение в виде $z = \beta_1(\log_q x) x,\, \beta_1\in\mathbb{C}[\log_q x]$, переходя к переменной $t=\log_q x$, получаем разностное 
уравнение на $\beta_1$:
$$\frac{1}{2}\beta_1(t+2) - 2 \beta_1(t+1) + \frac{3}{2}\beta_1(t) + 2a_3 = 0,$$ все решения которого в виде многочлена записываются в виде
$\beta_1(t) = 2 a_3 t + C,$ $C\in\mathbb{C}.$

Итого, получаем начальный отрезок формального разложения решения уравнения \eqref{eq:qP5} при 
$q=1/2:$
$$y(x) = -1 + (C - 2a_3\log_2 x)  \, x + \ldots, \quad C \in \mathbb{C}.$$

\textbf{Случай 2.} Пусть теперь $q=1/4$, тогда собственные  числа: $k_1 = 1/2, k_2 = -\log_4 (3/2)$: только $k_1$ критическое. 

Поскольку координаты всех точек носителя уравнения (\ref{eq:qP5-new}) целые 
неотрицательные, критическое число положительное полуцелое, а в носителе нет точки $(0, 0)$, то и 
множество $K(
k_1) \subset \mathbb{N}/2$. C другой стороны, в множестве $S'_+(\tilde{f})$  есть точки $(0, 1)$ и $(1/2, 
-1)$, их 
конечные суммы покрывают все числа вида $(n/2, -1), n\in\mathbb{N}$, значит, $K(
k_1) = \mathbb{N}/2$.

Найдем первый член разложения решения. Уравнение на $\beta_{1/2}(t)$ имеет вид 
$$\frac{1}{2}\beta_{1/2}(t+2)-2\beta_{1/2}(t+1)+\frac{3}{2}\beta_{1/2}(t)=0.$$
Для этого уравнения выполнено условие совместности. Все 
полиномиальные решения такого разностного уравнения являются константами: $\beta_{1/2}(t) 
=C, C\in \mathbb{C}$ -- произвольная постоянная. Найдем следующий член разложения:
разностное уравнение на $\beta_1$ имеет вид
$$\frac{1}{8}\beta_1(t+2) - \beta_1(t+1) + \frac{3}{2} \beta_1(t) = -2a_3-C^2/4.$$
Единственным полиномиальным решением данного разностного уравнения является $\beta_1(t) = 
-\frac{2}{5}(8a_3+C^2)$.

Получаем следующий начальный отрезок  разложения решения уравнения \eqref{eq:qP5}:
$$y(x) = -1 + C \sqrt{x} + -\frac{2}{5}(8a_3+C^2)x + \ldots$$

Итак, мы проиллюстрировали то, что в зависимости от значения параметра~$q$  уравнения степенное  укороченное решение алгебраического $q$-разностного уравнения (\ref{eq:qP5}) может быть продолжено как степенно-логарифмическое разложение  вида \eqref{eq:logsolution} при $q=1/2$ или степенное  разложение -- разложение в формальный ряд Пюизо -- при $q=1/4$.
\vskip 6cm

\textbf{Благодарности.} Статья подготовлена в ходе работы в рамках Программы фундаментальных исследований Национального исследовательского университета "Высшая школа экономики" (НИУ ВШЭ).
	
\vskip 1 cm
\textbf{Affiliations.}\\
HSE University,\\
%Moscow Institute of Electronics and Mathematics,\\
Tallinskaya 34, Moscow, 123458, Russia\\
e-mails: parus-a@mail.ru, gajanovnv@gmail.com

\end{document}